\documentclass[10pt,letterpaper]{amsart}
\usepackage{amsmath,amsfonts,amssymb,amsthm,array,stmaryrd,mathrsfs}
\usepackage[bookmarks=true]{hyperref}

\usepackage{amsrefs}

\newtheorem{thm}{Theorem}[section]
\newtheorem{lemma}{Lemma}[section]
\newtheorem{prop}[lemma]{Proposition}

\theoremstyle{definition}
\newtheorem{defn}[lemma]{Definition}
\theoremstyle{remark}
\newtheorem{remark}[lemma]{Remark}
\numberwithin{equation}{section}

\def\esmath{\ensuremath}

\newcommand{\mfr}[1]{\ensuremath \mathfrak{#1}}
\newcommand{\mbb}[1]{\ensuremath \mathbb{#1}}
\newcommand{\mbf}[1]{\ensuremath \mathbf{#1}}
\newcommand{\mcl}[1]{\ensuremath \mathcal{#1}}
\newcommand{\msc}[1]{\ensuremath \mathscr{#1}}

\newcommand{\mrm}[1]{\ensuremath \mathrm{#1}}

\def\NN{\esmath\mathbb N} %natural numbers
\def\RR{\esmath\mathbb R} %real numbers
\def\CC{\esmath\mathbb C} %complex numbers
\def\XX{\esmath\mathbb X}
\def\EE{\esmath\mathbb E}
\def\GG{\esmath\mathbb G}

\def\fh{\esmath\mathfrak h}

\newcommand{\p}{\partial}

\newcommand{\bn}{\begin{enumerate}}
\newcommand{\en}{\end{enumerate}}
\newcommand{\bi}{\begin{itemize}}
\newcommand{\ei}{\end{itemize}}
\newcommand{\bqq}{\begin{eqnarray*}}
\newcommand{\eqq}{\end{eqnarray*}}
\newcommand{\balg}{\begin{align*}}
\newcommand{\ealg}{\end{align*}}

\newcommand{\limt}[2]{\lim\limits_{{#1}\to{#2}}}

\newcommand{\sums}[2]{\sum\limits_{#1}^{#2}}

\DeclareMathOperator{\tr}{Tr}

\DeclareMathOperator{\spec}{Spec}

\DeclareMathOperator{\rc}{Ric}

\newcommand{\smfrac}[2]{{\textstyle{\frac{#1}{#2}}}}
\newcommand{\half}{{\smfrac{1}{2}}}
\DeclareMathOperator{\id}{id}

\newcommand{\hooklongrightarrow}{\lhook\joinrel\longrightarrow}

\begin{document}
%%%%%%%%%%%%%%%%%

\title[Dynamical stability of algebraic {R}icci solitons]{Dynamical stability of algebraic {R}icci solitons}

\author{Michael Bradford Williams}
\address{Department of Mathematics, University of California, Los Angeles}
\email{mwilliams@math.ucla.edu}
\urladdr{http://www.math.ucla.edu/~mwilliams/}

\author{Haotian Wu}
\address{Mathematical Sciences Research Institute, Berkeley, CA}
\email{hwu@msri.org}
\urladdr{http://www.ma.utexas.edu/users/hwu/}

%53C25    	Special Riemannian manifolds (Einstein, Sasakian, etc.)
%53C30     	Homogeneous manifolds
%53C44    	Geometric evolution equations (mean curvature flow, Ricci flow, etc.)
%58J37   	Perturbations; asymptotics

\date{\today}
%\date{last modified by \hwu{} \currenttime, on \usdate\today}

%%%%%%%%%%%%%%%%%
%\linenumbers
%%%%%%%%%%%%%%%%%

\begin{abstract} 
We consider dynamical stability for a modified Ricci flow equation whose stationary solutions include Einstein and Ricci soliton metrics.  Our focus is on homogeneous metrics on non-compact manifolds.  Following the program of Guenther, Isenberg, and Knopf, we define a class of weighted little H\"older spaces with certain interpolation properties that allow the use of maximal regularity theory and the application of a stability theorem of Simonett.  With this, we derive two stability theorems, one for a class of Einstein metrics and one for a class of non-Einstein Ricci solitons.  Using linear stability results of Jablonski, Petersen, and the first author, we obtain dynamical stability for many specific Einstein and Ricci soliton metrics on simply connected solvable Lie groups.
\end{abstract}

%%%%%%%%%%%%%%%%%
\maketitle
%%%%%%%%%%%%%%%%%

\section{Introduction}
This paper is motivated by the question of dynamical stability for stationary solutions of (suitably normalized) Ricci flow.  That is, given a stationary solution $g_0$ of Ricci flow and some topology on the space of metrics, does there exist a neighborhood $U$ of $g_0$ such that all Ricci flow solutions with initial data in $U$ converge to $g_0$?  Einstein metrics are examples of stationary solutions, and since the introduction of Ricci flow \cite{Hamilton1982}, many authors have considered the stability of Einstein metrics in a variety of contexts.  In the compact case,  Ye proved that Einstein metrics with certain curvature pinching properties are stable \cite{Ye1993}; Guenther, Isenberg, and Knopf proved that certain flat and Ricci-flat metrics are stable \cite{GuentherIsenbergKnopf2002}, and some of these results were improved by \v{S}e\v{s}um \cite{Sesum2006}; using the results of \v{S}e\v{s}um, Dai, Wang, and Wei proved that K\"{a}hler-Einstein metrics with non-positive scalar curvature are stable \cite{DaiWangWei2007}; Knopf and Young proved that hyperbolic space forms are stable \cite{KnopfYoung2009}.  In the context of Ricci coupled with other geometric flows, such as Yang-Mills flow and harmonic map flow, there are various results by Knopf \cite{Knopf2009}, Young \cite{Young2010}, and the first author \cite{Williams2013-systems}.   

In the non-compact setting, Schn\"urer, Schulze, and Simon have proven stability of Euclidean space and real hyperbolic space under a coupled Ricci flow system \cite{SchnurerEtAl2008,SchnurerEtAl2011}.  Results on the latter space were also obtained by Li and Yin \cite{LiYin2010}.  Bamler has proven stability of symmetric spaces of non-compact types, and also of finite-volume hyperbolic manifolds with cusps \cite{Bamler2010-sym,Bamler2010-cusps}.  The second author proved that complex hyperbolic space is stable \cite{Wu2013}.  We should note that these authors use various techniques and obtain stability relative to various topologies on the space of metrics.
 
Since Ricci solitons are fixed points of Ricci flow, modulo diffeomorphisms, we would like to expand the class of fixed point metrics under consideration to include non-Einstein solitons.  Our approach to proving stability follows the program initiated by Guenther, Isenberg, and Knopf in \cite{GuentherIsenbergKnopf2002}, which has several steps.  First, one must find a modified flow equation whose fixed points are those in question.  This flow is usually a rescaled version of Ricci flow.  Indeed, given $\lambda \in \mbb{R}$ and a vector field $X$ on a fixed manifold $M^n$, we consider the \textit{curvature-normalized Ricci flow}
\begin{equation}\label{eq:cnrf}
\partial_t g = -2 \rc(g) + 2\lambda g + \mcl{L}_X g.
\end{equation}
This has the feature that a Ricci soliton with
\begin{equation}\label{eq:ricci-soliton}
\rc(g) = \lambda g + \half \mcl{L}_X g
\end{equation}
is a stationary solution.  When $X$ is a Killing field (e.g., $X=0$), the metric is Einstein.

With a particular fixed point in mind, the next step is to compute the linearization of the flow at this fixed point, and to prove linear stability of the fixed point.  After modification by DeTurck diffeomorphisms, the flow \eqref{eq:cnrf} has linearization 
\begin{equation}\label{eq:cnrf-lin}
\partial_t h = \mbf{L} h := \Delta_L h + 2\lambda h + \mcl{L}_X h,
\end{equation}
where $\Delta_L h$ is the Lichnerowicz Laplacian acting on symmetric 2-tensors.  Recall that a fixed point is \textit{strictly} (resp.~\textit{weakly}) \textit{linearly stable} if there exists some $\epsilon >0$ (resp.~$\epsilon = 0$) such that\footnote{See Subsection 2.1 for notation.}
\begin{equation}\label{eq:lin-stab}
(\mbf{L} h,h) \leq - \epsilon \|h\|^2 
\end{equation}
for all non-zero symmetric covariant 2-tensors $h$ taken from some appropriate space of tensors.  This is equivalent to saying that the operator $\mbf{L}$ has negative (resp.~non-positive) spectrum.  

Finally, once linear stability has been established, one can obtain dynamical stability by appealing to various theorems from \textsc{pde} theory.  The idea of Guenther, Isenberg, and Knopf in  \cite{GuentherIsenbergKnopf2002} was to use the framework of maximal regularity theory as described by Da Prato and Grisvard \cite{DaPratoGrisvard1979}.  With respect to an appropriate sequence of spaces that possess nice interpolation properties, if one can show that the operator on the right hand side of \eqref{eq:cnrf} satisfies certain analytic properties, then one can apply a theorem of Simonett to obtain dynamical stability \cite{Simonett1995}.  (We note that Simonett's theorem allows for the presence of center manifolds, which occur when one only has weak linear stability; we will not need this.)  Since the work in \cite{GuentherIsenbergKnopf2002}, the methods described here have been used to prove a number of stability results \cite{Knopf2009,KnopfYoung2009,Young2010,Williams2013-systems,Wu2013}.

The goal of this paper is to use these techniques to study the dynamical stability of certain homogeneous Einstein and Ricci soliton metrics on non-compact manifolds.  The only non-Einstein homogeneous Ricci solitons occur on non-compact manifolds, and they must be expanding ($\lambda < 0$) and non-gradient ($X$ is not a gradient vector field); see Section 2 of \cite{Lauret2011-sol} for more details.  Furthermore, all known examples of these metrics can be realized as left-invariant metrics on simply connected solvable Lie groups, which satisfy
\begin{equation}\label{eq:alg-soliton}
\rc(g) = \lambda \id + D
\end{equation}
on the Lie algebra level.  Here $\rc$ is the Ricci endomorphism, $\lambda \in \mbb{R}$, and $D$ is a derivation of the Lie algebra.  A metric satisfying this equation is called an \textit{algebraic soliton}; in the case of a nilpotent or  solvable Lie groups, they are called \textit{nilsolitons} and \textit{solvsolitons}, respectively.  It turns out that every algebraic soliton is a Ricci soliton (see, e.g., \cite{Lauret2011-sol}), and conversely a left-invariant Ricci soliton on a solvable Lie group is isometric to a solvsoliton (on a possibly different Lie group).  Finally, if a solvable Lie group admits a non-flat algebraic soliton, then it must be diffeomorphic to $\mbb{R}^n$.  See \cite{Jablonski2011-hom} for a discussion of these last two facts.

The non-compactness of the manifolds in question introduces various analytical challenges to applying the stability techniques described above.  For example, in order to have non-compactly supported perturbations of the fixed point metric, we require the spaces of perturbations to be weighted in order to justify integration by parts.  Generalizing the weighted little H\"older space construction found in \cite{Wu2013} (such spaces are described in Section \ref{holder_space}), our result for Einstein metrics allows for quickly-decaying but non-compact perturbations.

\begin{thm}\label{thm:stab-einstein}
Let $(M^n,g)$ have the following properties:
\begin{itemize}
\item[(a)] $g$ is a strictly linearly stable Einstein metric satisfying $\rc(g) = \lambda g$;
\item[(b)] $(M,g)$ has infinite injectivity radius; 
\item[(c)] $M$ admits a single coordinate chart in which the Christoffel symbols of $g$ and their partial derivatives $\p^\ell \Gamma_{ij}^k$ are bounded by constants $C_{|\ell|}$, for $|\ell|=0, 1, 2$.
\end{itemize}
For each $\rho\in(0,1)$, there exists $\eta\in(\rho,1)$ such that the following is true.

There exists a neighborhood $\mathcal U$ of $g$ in the $\mathfrak h^{1+\eta}_\tau$-topology (cf.~Definition \ref{weighted-space}) such that for all initial data $\tilde g(0)\in\mathcal U$, the unique solution $\tilde g(t)$ of the curvature-normalized Ricci-DeTurck flow \eqref{eq:cnrdf} exists for all $t\geq 0$ and converges exponentially fast in the $\mathfrak h^{2+\rho}_\tau$-norm to $g$. 
\end{thm}

Another analytic issue is the operator on the right hand side of \eqref{eq:cnrf} (or, more precisely, \eqref{eq:cnrdf}). In the Einstein case the operator is reasonably nice, but in the soliton case it has unbounded first-order coefficients.  Indeed, after suitable interpretation (cf.~Lemma \ref{lem:coef-soliton}), it is an Ornstein-Uhlenbeck operator, and the challenges related to such operators---especially in the non-compact setting---are well-known.  To our knowledge, there are no dynamical stability theorems for non-compact Ricci solitons, so our second result---while perhaps non-optimal\footnote{Examples from bracket flow \cite{GlickPayne2010} suggest that one may have to ``fix infinity'' in some sense to get dynamical stability. It remains an interesting open question to determine a sufficient decay condition on the perturbations that will imply dynamical stability.}---provides a step in this direction (cf.~\cite{GuentherIsenbergKnopf2006}). Moreover, expanding homogeneous solitons are expected to act as singularity models for (at least some) Type-III Ricci flow singularities. This is known to be true in some cases, see \cite{Lott2010, Knopf2009}. Theorem \ref{thm:stab-soliton} exhibits many more examples of Ricci flow solutions that converge to expanding homogeneous solitons, thus contributing to the understanding of long-time behavior of Ricci flow solutions.

\begin{thm}\label{thm:stab-soliton}
Let $(M^n,g)$ have the following properties:
\begin{itemize}
\item[(a)] $M$ is a simply connected solvable Lie group;
\item[(b)] $g$ is a strictly linearly stable algebraic soliton metric satisfying $\rc(g) = \lambda \id + D$.
\end{itemize}
For each $R>0$, $\rho\in(0,1)$, there exists $\eta\in(\rho,1)$ such that the following is true.

There exists a neighborhood $\mathcal U$ of $g$ in the $\mathfrak h^{1+\eta}(B_R)$-topology (cf.~Definition \ref{unweighted-space}) such that for all initial data $\tilde g(0)\in\mathcal U$, the unique solution $\tilde g(t)$ of the curvature-normalized Ricci flow \eqref{eq:cnrf} exists for all $t\geq 0$ and converges exponentially fast in the $\mathfrak h^{2+\rho}(B_R)$-norm to $g$. 
\end{thm}

As previously mentioned, the focus of this paper is on the analytic aspects of dynamical stability---the last step in the outline above.  Linear aspects, i.e., proving linear stability, for algebraic solitons on simply connected solvable Lie groups were first described in \cite{GuentherIsenbergKnopf2006}, where the authors give a detailed analysis of three solitons:~ $\mrm{Nil}^3$, $\mrm{Sol}^3$, and $\mrm{Nil}^4$.  For more examples, we appeal to \cite{JablonskiEtAl2013-linear}, where the authors derive various conditions that guarantee linear stability and use them to give many examples of linearly stable metrics.  

\begin{thm}[\cite{JablonskiEtAl2013-linear}]\label{thm:linear}
The following algebraic solitons are strictly linearly stable with respect to the curvature-normalized Ricci flow:
\begin{enumerate}
\item\label{item1} every nilsoliton of dimension six or less, and every member of a certain one-parameter family of seven-dimensional nilsolitons;
\item every abelian or two-step nilsoliton;
\item every four-dimensional solvsoliton whose nilradical is the three-dimensional Heisenberg algebra;
\item an open set of solvsolitons whose nilradicals are codimension-one and abelian;
\item every solvable Einstein metric whose nilradical is codimension-one and 
\begin{enumerate}
\item found in \ref{item1} and has dimension greater than one, or
\item a generalized Heisenberg algebra, or
\item a free two-step nilpotent algebra;
\end{enumerate}
\item for each $m \geq 2$, an $(8m^2-6m-8)$-dimensional family of negatively-curved Einstein metrics containing the quaternionic hyperbolic space $\mbb{H}H^{m+1}$;
\item an $84$-dimensional family of negatively-curved Einstein metrics containing the Cayley hyperbolic plane $\mbb{C}aH^2$.
\end{enumerate}
Furthermore, any non-nilpotent solvsoliton on a Lie algebra $\mfr{s} = \mfr{n} \rtimes \mfr{a}$ satisfying $\mathring{R} < - \lambda$ and $0 < \dim(\mfr{a}) < \mrm{rank}(\mfr{n})$ is contained in an open set of stable solvsolitons.
\end{thm}

Here, $\mathring{R}$ is the action of the Riemann curvature tensor on symmetric 2-tensors, and $\mrm{rank}(\mfr{n})$ is the dimension of a maximal abelian subalgebra of symmetric derivations of $\mfr{n}$.  See \cite{JablonskiEtAl2013-linear} for more details.

Using Theorem \ref{thm:linear}, together with our dynamical stability theorems, we have the following result.

\begin{thm}\label{thm:dynamical}
Each metric appearing in Theorem \ref{thm:linear} is dynamically stable in the context of Theorem \ref{thm:stab-einstein} for Einstein metrics and of Theorem \ref{thm:stab-soliton} for non-Einstein solitons.
\end{thm}

\begin{remark}
Although compact quotients of simply connected solvable Lie groups do not admit algebraic soliton metrics (see for example, \cite[Theorem 1.4]{Jablonski2011-hom}), we do have stability of Einstein metrics on compact quotients, e.g., \cite[Theorem 1.1]{Wu2013}.
\end{remark}

This paper is organized as follows. In Section 2, we set up notation and define a large class of Riemannian manifolds relevant to our interest, and collect the linear stability results from \cite{JablonskiEtAl2013-linear}. In Section 3, we define little H\"{o}lder spaces on geodesic balls and suitably weighted little H\"{o}lder spaces on complete non-compact manifolds, and prove some embedding and interpolation properties of these spaces. In Section 4, we apply Simonett's stability theorem to prove the dynamical stability for linearly stable Einstein (Theorem \ref{thm:stab-einstein}) and Ricci soliton (Theorem \ref{thm:stab-soliton}) metrics. For completeness, we include Simonett's stability theorem in the appendix.

%%%%%%%%%%%%%%%%%%%%%%%%%%%%%%%%%%%%%%%%%%%%%%%%%%%%%

\section{Preliminaries}

%%%%%%%%%%%%%%%%%%%%%%%%

\subsection{Notation and convention}\label{notations}

We denote by $\mathcal S^2$ ($\mathcal S^2_c$, $\mathcal S^2_+$, resp.) the $\frac{n(n+1)}{2}$-dimensional vector space of \emph{symmetric} covariant 2-tensor fields (with compact support, positive-definite, resp.) over $M$. The regularity of the tensor fields will be either specified or dictated by context. We define $\Omega^1$ to be the space of 1-forms over $M$.

We denote by $g_0$ a fixed point of equation \eqref{eq:cnrf}, by $\mathcal L$ the Lie derivative, by 
\[ \delta=\delta_{g_0}:\mathcal S^2 \longrightarrow \Omega^1 \]
the divergence operator (with respect to $g_0$), by 
\[ \delta^\ast=\delta^\ast_{g_0}:\Omega^1 \longrightarrow \mathcal S^2 \]
the formal $L^2$-adjoint of $\delta$, by $d\mu_{g_0}$ the volume form of $g_0$. We use $\langle\cdot,\cdot\rangle$ to denote the tensor inner product with respect to $g_0$.  From this, we define the $L^2$-pairing on $\mathcal S^2_c$, denoted by $(\cdot,\cdot)$, by 
\begin{align*}
(h,k) := \int_{M}\langle h, k\rangle d\mu_{g_0}.
\end{align*}
We let $|h|^2:=\langle h,h\rangle$, $\Vert h\Vert^2:=(h,h)$; for a function $f$, $\Vert f\Vert^2$ is its $L^2$-norm. Finally, $\ell$ is a multi-index and $\nabla^\ell= \nabla^{\ell_1}_1\nabla^{\ell_2}_2\cdots\nabla^{\ell_n}_n$ for $|\ell| = \sums{i=1}{n}\ell_i$ (similar notation for $\p^\ell$).

\subsection{Geometric assumptions}
Throughout this paper, $(M, g)$ denotes a complete non-compact $n$-dimensional Riemannian manifold diffeomorphic to $\mbb{R}^n$ ($n\geq 1$). We use the expression ``$A\lesssim B$'' to mean $A\leq C B$, where $C>0$ is a constant that may change from line to line.

\begin{defn}\label{class_A}
We denote by $\msc A$ the class of Riemannian manifolds $(M, g)$ that have infinite injectivity radius and admit a single coordinate chart in which the Christoffel symbols of $g$ and their partial derivatives $\p^\ell \Gamma_{ij}^k$ are bounded by constants $C_{|\ell|}$ for $|\ell|=0, 1, 2$.
\end{defn}

Let $O\in M$ be identified with the origin of the single coordinate chart. We denote by $B_R\subset M$ the (closed) geodesic ball of radius $R$ centered at $O$ and by $V(R)$ its volume (with distance and volume both computed using $g$).

\begin{lemma}\label{f_tau}
Suppose $(M,g)\in\msc A$.  Then there exist constants $\tau, C >0$ and a monotonically increasing function $f_\tau:[0,\infty)\to[0,\infty)$ with $f(0)=0$ such that
\begin{align}\label{eq:f_tau(R)}
\sums{N=2}{\infty} \frac{V(2N)}{f_\tau(2N-2)} & < C.
\end{align}
\end{lemma}

\begin{proof}
We observe that $(M, g)\in\msc A$ has bounded sectional curvature $K_g$, i.e., there are constants $a, b$ such that $a\leq K_{g} \leq b$. Then by the Bishop volume comparison theorem, $V(R) \leq V_a(R)$ for all $R\geq0$, where $V_a(R)$ denotes the volume of a geodesic ball of radius $R$ in the space form of constant sectional curvature $a$. In particular, $a\leq 0$, for otherwise $(M, g)$ would be compact by the Bonnet-Myers theorem. If $a=0$, then $V_0(R)\lesssim R^n$, and we choose $f_\tau(R)=R^{n+\tau}$ for some $\tau>1$; if $a<0$, then $V_a(R)\lesssim e^{n R}$, and we choose $f_\tau(R) = e^{(n+\tau)R}$ for some $\tau>0$. It is then straightforward to verify \eqref{eq:f_tau(R)} in both cases.
\end{proof}

\begin{remark}
When $a<0$, the choice of $f_\tau$ in the proof above is not optimal. For example, $V(R)\lesssim e^{mR}$ on $(\mathbb{CH}^m,g_B)$, the complex hyperbolic space (of real dimension $2m$) with the Bergman metric (see, for example, \cite{Goldman1999}), and so we can choose $f_\tau(R)=e^{(m+\tau)R}$ for some $\tau>0$.
\end{remark}

%%%%%%%%%%%%%%%%%%%%%%%%%%%%%%%%%%%%%%%%%%%%%%%%%%
\subsection{Curvature-normalized {R}icci flow and linear stability}

In this subsection, we recall the construction of the curvature-normalized Ricci flow and describe its linearization.  We begin with the unnormalized Ricci flow equation on $M$,
\begin{equation}\label{eq:RF}
\partial_t g = -2\rc(g).
\end{equation}
In order to analyze Ricci solitons as stationary solutions of Ricci flow, Guenther, Isenberg, and Knopf modify this flow by scaling and diffeomorphisms. Indeed, given $\lambda < 0$ and a vector field $X$, they pull back by the diffeomorphisms generated by $\frac{2}{\lambda t} X$ and rescale time as $-\frac{2}{\lambda} \log t$ to obtain the flow
\[ \partial_t g = -2\rc(g) + 2 \lambda g + \mcl{L}_X g, \]
and a soliton satisfying \eqref{eq:ricci-soliton} with given $\lambda$ and $X$ is a stationary solution. See Section 3 of \cite{GuentherIsenbergKnopf2006} for more details.

Recall that the DeTurck trick is a modification of Ricci flow by certain diffeomorphisms. It was introduced as a means of enforcing strict parabolicity of Ricci flow and thereby allowing the proof of short-time existence of solutions by standard parabolic methods \cite{DeTurck1983}. From our perspective, the DeTurck trick also has the advantage of significantly simplifying the linearization of the curvature-normalized Ricci flow. Applying this trick to the curvature-normalized flow, we obtain the \textit{curvature-normalized Ricci--DeTurck flow},
\begin{equation}\label{eq:cnrdf}
\partial_t g = Q(g) : = -2\rc(g)+ 2 \lambda g + \mcl{L}_X g - P_{g_0}(g),
\end{equation}
where $P_{g_0}(g):=-2\delta^\ast(\tilde g_0 \delta(G(g,g_0)))$, with $g_0\in\mathcal S^2_+$, $G(g,g_0):=g_0-\frac{1}{2}(\tr_g g_0)g$, and the map $\tilde g_0:\Omega^1\to\Omega^1$ given by $(\tilde g_0 \beta)_j := g_{jk}(g_0)^{k i}\beta_i$. In particular, if $g_0$ is a fixed point of \eqref{eq:cnrf}, then $g_0$ is also a fixed point of \eqref{eq:cnrdf}; see \cite{GuentherIsenbergKnopf2002}.

Now, for $h\in\mathcal S^2$ sufficiently differentiable, standard computations (see, e.g., \cite{GuentherIsenbergKnopf2006}) show that the linearization of this flow is precisely
\[ \partial_t h = \mbf{L} h: = \Delta_L h + 2\lambda h + \mcl{L}_X h, \]
where $\mbf{L}:\mcl S^2 \to \mcl S^2$ is a linear operator and $\Delta_L$ is the Lichnerowicz Laplacian.

A crucial step in the stability program is to establish the linear stability of a fixed point, i.e., to verify \eqref{eq:lin-stab}.  For specific algebraic solitons, we appeal to \cite{JablonskiEtAl2013-linear}, which shows that about one hundred individual examples and several infinite families of metrics are strictly linearly stable.  Moreover, all the examples in Theorem \ref{thm:linear} above belong to the class $\msc A$.  Indeed, any non-flat Ricci soliton with a transitive solvable group of isometries is diffeomorphic to $\mbb{R}^n$; see \cite{Jablonski2011-hom}.  Now Theorem \ref{thm:dynamical} follows from Theorems \ref{thm:stab-einstein} and \ref{thm:stab-soliton}.

%%%%%%%%%%%%%%%%%%%%%%%%%%%%%%%%%%%%%%%%%%%%%%%%%

\section{(Weighted) little H\"{o}lder spaces}\label{holder_space}

In this section, we recall the definition of little H\"{o}lder spaces on compact manifolds and define suitably weighted little H\"{o}lder spaces on complete non-compact manifolds. We then prove some embedding and interpolation properties of these spaces.

\begin{remark}
Similar definitions and results have appeared in \cite{Wu2013}. In below, we adopt the notation in \cite{Wu2013} whenever convenient.
\end{remark}

%%%%%%%%%%%%%%%%%%%%%%%%%%%%%%%%%%%%%%%%%%%%%%%%%

\subsection{Unweighted spaces on geodesic balls}
Suppose for now that $(M, g_0)$ is \emph{any} complete non-compact Riemannian manifold and $B_R$ is a (closed) geodesic ball of radius $R$ (with distance computed using $g_0$). For each $R>0$, consider any smooth $h\in\mathcal{S}^2_c(B_R)$, i.e., $h$ is compactly supported on $B_R$. Fixing a background metric $\hat g$ and a finite collection $\{ U_\upsilon \}_{1\leq\upsilon\leq\Upsilon}$ of coordinate charts covering $B_R$, we denote
\begin{align*}
 [h_{ij}]_{k+\alpha; U_\upsilon} := \sup\limits_{|\ell|=k} \sup\limits_{x\neq y \in U_\upsilon} \frac{\left|\nabla^\ell h_{ij}(x)- \nabla^\ell h_{ij}(y)\right|}{d_{\hat g}(x,y)^\alpha},
\end{align*}
where $d_{\hat g}$ is the distance function for $\hat g$.

\begin{defn}\label{unweighted-space}
For each integer $k\geq 0$ and $\alpha\in(0,1)$, the $\Vert \cdot\Vert_{k+\alpha;B_R}$-norm of $h$ is defined by
\begin{align*}
\Vert h \Vert_{k+\alpha; B_R} := 
\sup_{\substack{1\leq i,j\leq n,\\ 1\leq\upsilon\leq\Upsilon}}\left(\sum\limits_{m=0}^{k}\sup\limits_{|\ell|=m} \sup\limits_{x\in U_\upsilon}\left|\nabla^\ell h_{ij}(x)\right|+[h_{ij}]_{k+\alpha; U_\upsilon}\right).
\end{align*}
The (unweighted) \emph{little H\"{o}lder space $\mfr{h}^{k+\alpha}(B_R)$} is the closure of smooth symmetric covariant 2-tensor fields with compact support on $B_R$ in the $\Vert\cdot \Vert_{k+\alpha;B_R}$-norm. 
\end{defn}

\begin{remark}
Note that $h\in\mfr{h}^{k+\alpha}(B_R)$ vanishes on the boundary $\p B_R$, this justifies integration by parts when proving linear stability for perturbations in $\mfr{h}^{k+\alpha}(B_R)$.
\end{remark}
\begin{remark}
The compactness of $B_R$ implies that different choices of background metrics or coordinate charts give equivalent norms. In particular, when proving dynamical stability of Ricci solitons under compactly supported perturbations, we can work with the coordinate chart given in Lemma \ref{X_0}.
\end{remark}

%%%%%%%%%%%%%%%%%%%%%%%%%%%%%%%%

\subsection{Weighted spaces on non-compact manifolds}
When defining the H\"{o}lder norm of a function (or tensor) on a complete non-compact Riemannian manifold, the particular choice of the atlas and the background metric matters. We generalize the construction in \cite{Wu2013}. Let $(M,g_0)\in\msc A$. In application, $g_0$ will be a fixed point of equation \eqref{eq:cnrf}. By Definition \ref{class_A}, we will work in a single coordinate chart. We set the background metric to be $g_0$ and compute $|\cdot|$, $\Vert\cdot\Vert$, and $d$ using $g_0$.

Let $h\in\mathcal S^2$ be smooth and possibly without compact support. Fix $n\in\NN$, consider the open covering of $M$ by the family $\{A_N\}_{N\in\NN}$ of overlapping open annuli and an open disk defined by 
\begin{align*}
A_1:=\{x:d(x,O)<4\},\quad A_N:=\{x:N-1<d(x,O)<N+3\}\quad\text{for }N\geq 2.
\end{align*}
 If $x,y\in A_N$, we denote $d_{x; A_N}:=d(x,\p A_N)$, $d_{x,y; A_N}:=\min\{d_{x;A_N},d_{y;A_N}\}$, and write $d_x, d_{x,y}$ whenever there is no ambiguity.

Given a smooth $h\in\mathcal{S}^2$ over $M$, for integers $k,q\geq 0$, multi-index $\ell$, and $\alpha\in(0,1)$, we let
\begin{align*}
 |h_{ij}|'_{q; A_N} & := \sup\limits_{|\ell|=q} \sup_{x\in A_N} d_x^q |\p^\ell h_{ij}(x)|,\\ 
 [h_{ij}]'_{k+\alpha; A_N} & := \sup\limits_{|\ell|=k} \sup_{x\neq y\in A_N} d_{x,y}^{k+\alpha} \frac{|\p^\ell h_{ij}(x)- \p^\ell h_{ij}(y)|}{d(x,y)^\alpha}.
\end{align*}

\begin{defn}\label{weighted-space}
Let $\tau>0$ and let the function $f_\tau$ be given as in Lemma \ref{f_tau}. The \emph{$\tau$-weighted}\footnote{The relevant parameter is $\tau$ as $f_\tau$ is canonically determined by $g_0$, cf.~the proof of Lemma \ref{f_tau}.} $\Vert\cdot\Vert_{k+\alpha;\tau}$-norm of $h\in\mathcal{S}^2$ ($h$ not necessarily compactly supported) is defined by
\begin{align}\label{eq:weighted_norm}
\Vert h \Vert_{k+\alpha;\tau} := 
\sup_{1\leq i,j\leq n} \sup_{N\in\NN}\; \left[f_\tau(N)\right]^{1/2} \left(\sum\limits_{q=0}^{k}|h_{ij}|'_{q; A_N}
+[h_{ij}]'_{k+\alpha; A_N}\right).
\end{align}
The \emph{$\tau$-weighted little H\"{o}lder space $\mathfrak h^{k+\alpha}_\tau := \mathfrak{h}^{k+\alpha}(M)$} on a complete non-compact manifold $M$ is the closure of smooth symmetric covariant 2-tensor fields compactly supported on $M$ in the $\Vert\cdot \Vert_{k+\alpha;\tau}$-norm.
\end{defn}

\begin{remark}
Definition \ref{weighted-space} is a generalization of that in \cite[Section 4]{Wu2013}. In \cite{Wu2013}, the underlying manifold is $(\mathbb{CH}^m, g_B)$ for which we simply set $f_\tau(R)=e^{(m+\tau)R}$ for some $\tau>0$. In the present paper, we allow the underlying manifold to be any member in a much more general class $\msc{A}$ (see Definition \ref{class_A}), and we prove in Lemma \ref{f_tau} that on such a manifold there exists a function $f_\tau$ that satisfies a growth condition \eqref{eq:f_tau(R)}.
\end{remark}

The following lemma concerns an infinitesimal property of $h\in \fh^{k+\alpha}_\tau$.
\begin{lemma}\label{Fht}
If $h\in\fh^{k+\alpha}_\tau$, define
\begin{align*}
 F_{h}(t) := \sup\limits_{1\leq i,j\leq n}\sup\limits_{N\in\NN}\; \left[f_\tau(N)\right]^{1/2}\left(\sup\limits_{|\ell|=k} \sup_{\substack{x\neq y\in A_N,\\ d(x,y)\leq t}} d_{x,y}^{k+\alpha}\frac{|\p^\ell h_{ij}(x)- \p^\ell h_{ij}(y)|}{d(x,y)^\alpha}\right).
\end{align*}
Then $\limt{t}{0+} F_h(t) = 0$.
\end{lemma}
\begin{proof}
The proof is the same as that of \cite[Lemma 4.1]{Wu2013}.
\end{proof}
\begin{remark}
From now on, we shorten some expressions. For integer $k\geq0$, $\nabla^k$ (or $\p^k$) means first applying $\nabla^\ell$ (or $\p^\ell$) and then taking the supremum over $|\ell|=k$. We also omit the indices of $h$, so then we must sum or take the supremum over $1\leq i,j\leq n$.
\end{remark}

%%%%%%%%%%%%%%%%%%%%%%%%

\subsection{Properties of weighted spaces}
We denote by $\mathcal W_1^2$ the Sobolev space of symmetric covariant 2-tensor fields $h$ over $M$ such that
\begin{align*}
 \sums{i,j=1}{n} \int_{M} \left( |h_{ij}|^2+|\nabla h_{ij}|^2 \right) d\mu_{g_0} < \infty,
\end{align*}
where $\nabla$ is computed using the background metric $g_0$. Denoting by $\overset{d}\hooklongrightarrow$ a continuous and dense embedding, we have the following embedding result.

\begin{prop}\label{dom}
Let $k, n\in\NN$, $\alpha\in(0,1)$, and fix $\tau>0$. If $(M, g_0)\in\msc A$, then 
\[ \mathfrak h^{k+\alpha}_\tau \overset{d}\hooklongrightarrow \mathcal W_1^2. \]
\end{prop}
\begin{proof}[Proof]
Since $(M, g_0)\in\msc A$ is a smooth complete Riemannian manifold, smooth tensor fields with compact support are dense in $\mathcal W_1^2$; see \cite{Aubin1982, Hebey1999}. So it suffices to show $\mathfrak h^{k+\alpha}_\tau \hooklongrightarrow \mathcal W_1^2$.

Let $h\in \mathfrak h^{k+\alpha}_\tau$. In the single coordinate chart on $(M, g_0)\in\msc A$, $\nabla h = \p h + \Gamma\ast h$, where $\Gamma$ is uniformly bounded by some constant $C$. If $x\in B_2$, then $x\in A_1$ with $d_x\geq 1$, so then
\begin{align*}
\int_{B_2} \left(|h(x)|^2 + |\nabla h(x)|^2 \right)d\mu_{g_0}(x) & \lesssim \int_{B_2} \left( |h(x)|^2+|\p h(x)|^2 \right) d\mu_{g_0}(x)\\
& \leq \int_{B_2} \left[ (|h(x)|'_{0;A_1})^2+d_x^{-2}(|h(x)|'_{1;A_1})^2 \right] d\mu_{g_0}(x)\\
& \leq \frac{V(2)}{f_\tau(1)} \Vert h\Vert^2_{k+\alpha;\tau}.
\end{align*}
Similarly, if $x\in B_{2N}\setminus B_{2N-2}$ ($N\geq 2$), then $x\in A_{2N-2}$ with $d_x\geq 1$. So 
 \[ |h(x)|\leq |h|'_{0;A_{2N-2}}, \qquad
|\p h(x)|\leq d_x^{-1}|h|'_{1;A_{2N-2}}\leq |h|'_{1;A_{2N-2}}, \]
and hence
\begin{align*}
\int_{B_{2N}\setminus B_{2N-2}} \left(|h(x)|^2 + |\nabla h(x)|^2 \right) d\mu_{g_0}(x) & \lesssim \frac{V(2N)}{f_\tau(2N-2)}\Vert h\Vert^2_{k+\alpha;\tau}.
\end{align*}
Then, with $B_0=\emptyset$, we obtain
\begin{align*}
\int_{B_{2N}}\left( |h(x)|^2+|\nabla h(x)|^2 \right) d\mu_{g_0}(x) &= \sums{K=1}{N}\int_{B_{2K}\setminus B_{2K-2}} \left( |h(x)|^2+|\nabla h(x)|^2 \right) d\mu_{g_0}(x)\\
& \lesssim \Vert h\Vert^2_{k+\alpha;\tau} \left(\frac{V(2)}{f_\tau(1)} + \sums{K=2}{N} \frac{V(2K)}{f_\tau(2K-2)}\right)\\
& \lesssim \Vert h\Vert^2_{k+\alpha;\tau},
\end{align*}
where the last inequality follows from \eqref{eq:f_tau(R)}. Letting $N\to\infty$, we conclude $h\in \mathcal W_1^2$.
\end{proof}

Proposition \ref{dom} allows us to integrate by parts, and hence extend the linear stability to perturbations that are not necessarily compactly supported.

\begin{lemma}\label{eqholds}
Let $k, n\in\NN$, $\alpha\in(0,1)$, and fix $\tau>0$. For $(M, g_0)\in\msc A$, the strict linear stability \eqref{eq:lin-stab} holds for all $h\in \mathfrak h^{k+\alpha}_\tau\setminus\{0\}$.
\end{lemma}
\begin{proof}
By Proposition \ref{dom}, $\mathfrak h^{k+\alpha}_\tau\overset{d}\hooklongrightarrow \mathcal W_1^2(M)$, so the lemma follows by strong convergence in the $\mathcal W^2_1$-norm.
\end{proof}

Given two Banach spaces $X, Y$ with $Y\overset{d}\hooklongrightarrow X$, for every $h\in X+Y$ and $t>0$, we define
\begin{align*}
K(t,h, X, Y) := \inf_{\substack{h=a+b,\\ a\in X,\,b\in Y}}\left(\Vert a\Vert_X+t\Vert b\Vert_Y \right).
\end{align*}
For $\theta\in(0,1)$, the \emph{continuous interpolation space} between $X$ and $Y$ is defined by
\begin{align*}
 (X,Y)_{\theta}:=\{h\in X+Y:\limt{t}{0+} t^{-\theta} K(t,h, X, Y)=0\}
\end{align*}
endowed with the norm $\Vert h\Vert_{\theta} := \Vert t^{-\theta} K(t,h, X, Y)\Vert_{L^\infty}$; see \cite{Amann1995, Lunardi2009, Simonett1995}.

By a standard fact in interpolation theory, e.g., \cite[Corollary 1.7]{Lunardi2009}, we have the following lemma.
\begin{lemma}\label{a7}
Fix $\tau>0$. Let $0\leq k \leq l \leq 2$ be integers, $0<\alpha\leq\beta<1$, and $0<\theta<1$, then there exists a constant $C(\theta)>0$ depending on $\theta$ such that for all $h\in \mathfrak h^{l+\beta}_\tau$,
\begin{align}\label{eq:a7-1}
 \Vert h \Vert_{(\mathfrak h^{k+\alpha}_\tau,\; \mathfrak h^{l+\beta}_\tau)_\theta} \leq C(\theta) \Vert h \Vert^{1-\theta}_{\mathfrak h^{k+\alpha}_\tau}\Vert h \Vert^{\theta}_{\mathfrak h^{l+\beta}_\tau}.
\end{align}
\end{lemma}

We also have an interpolation result that generalizes \cite[Theorem 4.1]{Wu2013}.
\begin{thm}\label{interp}
Under the assumptions of Lemma \ref{a7}, if $(1-\theta)(k+\alpha)+\theta(l+\beta)\notin\NN$, then there is a Banach space isomorphism
\begin{align*}
(\mathfrak h^{k+\alpha}_{\tau},\mathfrak h^{l+\beta}_{\tau})_{\theta}\cong \mathfrak h^{(1-\theta)(k+\alpha)+\theta(l+\beta)}_{\tau},
\end{align*}
with equivalence of the respective norms.
\end{thm}
\begin{proof}
Close examination of the proof of \cite[Theorem 4.1]{Wu2013} shows that we only use the following facts: the boundedness of sectional curvatures of $g_0$, the existence of a single coordinate chart in which $\p^\ell \Gamma^{i}_{jk}$ are bounded by $C_{|\ell|}$ for $|\ell| = 0,1,2$, and Lemma \ref{Fht}. Then the proof for $(M, g_0)\in\mathscr{A}$ is a straightforward modification of that in \cite[Section 7]{Wu2013}, and we omit the details here.
\end{proof}

\begin{remark}
Theorem \ref{interp} also holds for the spaces $\mfr h^{k+\alpha}(B_R)$, $\mfr h^{l+\beta}(B_R)$.
\end{remark}

\begin{remark}
If we let $|\ell|=0,1,\ldots, L$ in Definition \ref{class_A}, then we can set $l \leq L$ in Lemma \ref{a7} and Theorem \ref{interp}.
\end{remark}

Consequently, $\mathfrak h^{l+\beta}_\tau\overset{d}\hooklongrightarrow (\mathfrak h^{k+\alpha}_\tau,\mathfrak h^{l+\beta}_\tau)_\theta \overset{d}\hooklongrightarrow \mathfrak h^{k+\alpha}_\tau$ for $l\geq k$, $\beta\geq\alpha$.

%%%%%%%%%%%%%%%%%%%%%%%%%%%%%%

\section{Dynamical stability}
Let $(M, g_0)\in\msc A$, and assume that $g_0$ is a fixed point of the curvature-normalized Ricci flow \eqref{eq:cnrf}. Then $g_0$ is also a fixed point of the curvature-normalized Ricci-DeTurck flow \eqref{eq:cnrdf} with the DeTurck term $P_{g_0}(g)$. If $g_0$ is linearly stable, then we will show that $g_0$ is also dynamically stable.

\begin{remark}
The strategy of proof is the same as that in \cite{Wu2013}. Henceforth, we again follow the notation in \cite{Wu2013}.
\end{remark}

%%%%%%%%%%%%%%%%%%%%%%%%%%%%%%

\subsection{Dynamical stability of an Einstein metric} 
When $X$ is a Killing field in equation \eqref{eq:cnrdf}, the fixed point $g_0$ is Einstein. In this subsection, the quantities $|\cdot|$, $\Vert\cdot\Vert$, $\Gamma_{ij}^k$, and the distance function $d$ are computed using the fixed background metric $g_0$, and $f_\tau$ is given in Lemma \ref{f_tau}. Our goal is to prove Theorem \ref{thm:stab-einstein}.

Fix $0<\sigma<\rho<1$ and consider the following sequence of densely embedded spaces:
\begin{align*}
\begin{array}{ccccccc}
\mathbb X_1 & \overset{d}\hooklongrightarrow & \mathbb E_1 & \overset{d}\hooklongrightarrow &\mathbb X_0 & \overset{d}\hooklongrightarrow & \mathbb 
E_0\\
\|	& & \| & & \| & & \| \\
\mathfrak h^{2+\rho}_\tau & & \mathfrak h^{2+\sigma}_\tau & &\mathfrak h^{0+\rho}_\tau& & \mathfrak h^{0+\sigma}_\tau\\
\end{array}
\end{align*}
For fixed $\frac{1}{2}\leq\beta<\alpha<1-\frac{\rho}{2}$, define
\begin{align*}
\mathbb X_\alpha:=(\mathbb X_0,\mathbb X_1)_\alpha
\quad \text{and} \quad 
\mathbb X_\beta:=(\mathbb X_0,\mathbb X_1)_\beta.
\end{align*}
Then by Theorem \ref{interp}, 
\[ \XX_\alpha\cong\mathfrak h^{2\alpha+\rho}_\tau
\quad \text{and} \quad
\XX_\beta\cong \mathfrak h^{2\beta+\rho}_\tau \]
with equivalence of the respective norms. Note that 
\[ \XX_\alpha \overset{d}\hooklongrightarrow \XX_\beta \overset{d}\hooklongrightarrow \mathfrak h^{1+\rho}_\tau. \]

The modified Ricci flow \eqref{eq:cnrdf} can be rewritten as
\begin{align*}
\p_t g = Q_{g_0}(g)g,
\end{align*}
where $Q_{g_0}(g)$ is a quasilinear elliptic operator.  The following lemma is proved in \cite[Lemma 3.1]{GuentherIsenbergKnopf2002} or \cite[Lemma 5.1]{Wu2013}.

\begin{lemma}\label{coef}
If we express $Q_{g_0}(g)g$ in terms of the first and second derivatives of the Einstein metric $g$ in local coordinates, then
\begin{align}\label{eq:coef}
(Q_{g_0}(g)g)_{ij}& =a\left(g_0(x),g_0^{-1}(x),g(x),g^{-1}(x)\right)_{ij}^{k l pq}\p_p\p_qg_{k l}\\
&\quad + b\left(g_0(x),g_0^{-1}(x),\p g_0(x),g(x),g^{-1}(x),\p g(x)\right)_{ij}^{k l p}\p_p g_{kl} \notag \\
&\quad\quad +c\left(g_0^{-1}(x),\p g_0(x),\p^2 g_0(x),g(x)\right)_{ij}^{kl}g_{kl}. \notag
\end{align}
The coefficients $a, b, c$ are analytic functions of their arguments and depend smoothly on $x\in M$.
\end{lemma}

For a fixed $\hat g\in \XX_\beta$, we can view $Q_{g_0}(\hat g)$ as a linear operator on $\XX_1$. If $h\in\XX_1$, then
\begin{align*}
(Q_{g_0}(\hat g)h)_{ij}& =a\left(g_0(x),g_0^{-1}(x),\hat g(x),\hat g^{-1}(x)\right)_{ij}^{kl pq}\p_p\p_q h_{kl}\\
&\quad + b\left(g_0(x),g_0^{-1}(x),\p g_0(x),\hat g(x), \hat g^{-1}(x),\p \hat g(x)\right)_{ij}^{kl p}\p_p h_{kl} \\
&\quad\quad +c\left(g_0^{-1}(x),\p g_0(x),\p^2 g_0(x), \hat g(x)\right)_{ij}^{kl} h_{kl}.
\end{align*}
In fact, $Q_{g_0}(\hat g)$ is a second order elliptic operator with bounded coefficients. Since the coefficient $a$ depends analytically on $\hat g$, if $\hat g$ is sufficiently close to $g_0$ in $\XX_\beta$, then $Q_{g_0}(\hat g)$ will remain elliptic with bounded coefficients. We call such a $\hat g$ an \emph{admissible perturbation} of $g_0$.

For $0<\epsilon\ll 1$ to be chosen in Lemma \ref{analytic}, we define an open set in $\XX_\beta$ by
\begin{align*}
\mathbb \GG_\beta := \{g\in\mathbb X_\beta \text{ is an admissible perturbation}: g>\epsilon g_0\},
\end{align*}
where ``$g>\epsilon g_0$'' means $g(X,X)>\epsilon g_0(X,X)$ for any tangent vector $X$. We also define
\begin{align*}
\mathbb \GG_\alpha := \mathbb \GG_\beta\cap\mathbb X_\alpha.
\end{align*}
 
We denote by $ L_{\hat g}:=Q_{g_0}(\hat g)$ the unbounded linear operator on $\mathbb X_0$ with dense domain $D(L_{\hat g}):=\mathbb X_1$. We extend $L_{\hat g}$ to $\hat L_{\hat g}$, which is now defined on $\mathbb E_0$ with dense domain $D(\hat L_{\hat g}):=\EE_1$. If $X,Y$ are two Banach spaces, we denote by $\mathcal L(X,Y)$ the space of bounded linear operators from $X$ to $Y$.

\begin{lemma}\label{bddop1}
The operators $L_{\hat g}, \hat L_{\hat g}$ satisfy the following properties.
\bn
\item $\hat g\mapsto L_{\hat g}$ is an analytic map $\GG_\beta \to \mathcal L(\XX_1,\XX_0)$.
\item $\hat g\mapsto \hat L_{\hat g}$ is an analytic map $\GG_\alpha \to \mathcal L(\EE_1,\EE_0)$.
\en 
\end{lemma}
\begin{proof}
The proof is identical to that of \cite[Lemma 6.1]{Wu2013}.
\end{proof}

\begin{defn}\label{sect_op}
Given a Banach space $X$, a linear operator $A:D(A)\subset X\to X$ is called \emph{sectorial} if there are constants $\omega\in\RR$, $\beta\in(\pi/2, \pi)$, and $C>0$ such that 
\bn
\item The resolvent set $\rho(A)$ contains 
\begin{align*}
S_{\beta,\omega}:=\{\eta\in\CC:\eta\neq\omega,|\arg(\eta-\omega)|<\beta\}, 
\end{align*}
\item $\Vert (\eta I-A)^{-1}\Vert_{\mathcal L(X,X)}\leq \frac{C}{|\eta-\omega|}, \text{ for all }\eta\in S_{\beta,\omega}$.
\en
\end{defn}

We also have the following lemma.

\begin{lemma}\label{analytic}
The operator $\hat L_{g_0}$ is sectorial, and there exists $\epsilon >0$ in the definition of $\GG_\beta$ such that for each $\hat g \in \GG_\alpha$, the operator $\hat L_{\hat g}$ generates a strongly continuous analytic semigroup on $\mathcal L(\EE_0,\EE_0)$.
\end{lemma}

\begin{proof}
We first claim that the linear operators $\hat{L}_{g_0}$ and $\mbf{L}$ differ by a bounded amount. Indeed, recall that $Q(g)$ is the quasilinear operator on the right hand side of equation \eqref{eq:cnrdf}, and $Q(g_0)=0$ since $g_0$ is a fixed point. Let $g=g_0+\tilde h$, where $\Vert \tilde h\Vert_{\XX_\alpha}\leq 1$ and $\tilde h$ is assumed to be smooth. Then we have by linearization
\begin{align*}
Q(g_0 + \tilde h) & = Q(g_0) + DQ|_{g_0} (\tilde h) + O\left(\Vert \tilde h \Vert^2\right)\\
	 		& = \mbf{L} (\tilde h) + O\left(\Vert \tilde h \Vert^2\right),
\end{align*}
and by freezing the coefficients
\begin{align*}
Q(g_0 + \tilde h)=Q_{g_0}(g)(g_0+ \tilde h)=\hat L_g (g_0) + \hat L_{g}(\tilde h).
\end{align*}
It follows that
\begin{align*}
\hat L_g \tilde h & = \mbf{L} (\tilde h) - \hat L_{g} (g_0) + O\left(\Vert\tilde h\Vert^2\right)\\
& = \mbf{L}(\tilde h) + O\left(\Vert\tilde h\Vert^2\right)
\end{align*}
since $ \hat L_{g_0} (g_0) = Q_{g_0}(g_0)(g_0)=Q(g_0)=0$. So there exists a constant $C>0$ such that
\begin{align*}
\mbf{L}(\tilde h) - C \tilde h \leq \hat L_g (\tilde h) \leq \mbf{L}(\tilde h) + C \tilde h,
\end{align*}
which proves the claim. Then the linear stability of $\mbf{L}$ (cf.~\eqref{eq:lin-stab}) implies the spectrum bound $\spec(\hat L_{g_0})\subset (-\infty,\eta_0)$ for some $\eta_0\in\RR$, and hence $\eta I - \hat L_{g_0}$ is a topological linear isomorphism from $\EE_1$ onto $\EE_0$ whenever $\Re(\eta) \geq \eta_0$.

Since $\hat L_{g_0}$ is a linear second order elliptic operator with bounded coefficients, the standard Schauder estimates for $\hat L_{g_B}$ on $A_N$ imply that
\begin{align*}
 \sum\limits_{\ell=0}^{2}|h|'_{\ell; A_N}+[h]'_{2+\sigma; A_N} \leq C\left ( |\hat L_{g_B} h|'_{0; A_N}+[\hat L_{g_B} h]'_{\sigma; A_N} + |h|'_{0;A_N} \right),
\end{align*}
where $C$ is a constant independent of $N$; see \cite{GilbargTrudinger2001}. Multiplying both sides of this inequality by $\left[f_\tau(N)\right]^{1/2}$ and taking the supremum over $N\in\NN$, we have
\begin{align*}
 \Vert h \Vert_{2+\sigma;\tau} & \lesssim \Vert \hat L_{g_0} h\Vert_{0+\sigma;\tau} + \Vert h \Vert_{0;\tau}.
\end{align*}
Then for every $h\in\EE_1=D(\hat L_{g_0})$, applying the Schauder estimates to the operator $\eta I - \hat{L}_{g_0}$, we have
\begin{align*}
\Vert h \Vert_{2+\sigma;\tau} & \lesssim \Vert (\eta I - \hat L_{g_0}) h\Vert_{0+\sigma;\tau},
\end{align*}
where we have used $\mfr h^{0+\sigma}_\tau\overset{d}\hooklongrightarrow \mfr h^{0}_{\tau}$ and $(\eta I - \hat L_{g_0})^{-1}\in\mathcal L(\EE_0, \EE_0)$.
This suffices to establish that $\hat L_{g_0}$ is sectorial by \cite[Theorem 1.2.2 and Remark 1.2.1 (a)]{Amann1995} (also see \cite[Lemma 3.4]{GuentherIsenbergKnopf2002}). Furthermore, 
$\hat L_{g_0}$ is densely defined by construction, so $\hat L_{g_0}$ generates a strongly continuous analytic semigroup by a standard characterization \cite[p.~34]{Lunardi1995}.

By part (2) of Lemma \ref{bddop1}, we can choose $\epsilon >0$ in the definition of $\GG_\beta$ so small that for $\hat g\in \GG_\alpha$, we have
\begin{align*}
 \Vert \hat L_{\hat g}-\hat L_{g_0} \Vert_{\mathcal L(\EE_1,\EE_0)} < \frac{1}{C+1}
\end{align*}
for the constant $C>0$ in the definition of sectorial operator corresponding to $\hat L_{g_0}$. So the perturbation $\hat L_{\hat g}$ is sectorial by \cite[Proposition 2.4.2]{Lunardi1995}, and hence generates a strongly continuous analytic semigroup on $\mathcal L(\EE_0,\EE_0)$.
\end{proof}

\begin{proof}[Proof of Theorem \ref{thm:stab-einstein}]
Using Lemmata \ref{bddop1} and \ref{analytic}, Theorem \ref{thm:stab-einstein} follows from Simonett's stability theorem (Theorem \ref{stab}) by the same argument as that for \cite[Theorem 1.2]{Wu2013}.
\end{proof}

%%%%%%%%%%%%%%%%%%%%%%%%%%%%%

\subsection{Dynamical stability of an algebraic Ricci soliton}\label{subsec:dynamical-soliton} 

Now we study an algebraic Ricci soliton on a simply connected solvable Lie group, $(S, g_0)$, where $\dim S = n$.  On the Lie algebra level, this satisfies
\[ \rc(g_0) = \lambda \, \id + D \]
for some $\lambda < 0$ and $D \in \mrm{Der}(\mfr{s})$.  This condition implies the existence of a vector field $X_0$ such that $g_0$ is a Ricci soliton in the usual sense (e.g., see Section 2 of \cite{Lauret2011-sol}):
\[ \rc(g_0) = \lambda g_0 + \half \mcl{L}_{X_0} g_0. \]

We have the following lemma on the growth of the vector field $X_0$.
\begin{lemma}
The vector field $X_0$ associated to an algebraic soliton grows linearly.
\end{lemma}

\begin{proof}
The soliton equation implies that for any vector field $Y \in C^\infty(TS)$, we have
\begin{align*}
 \rc(g_0)(Y,Y) - \lambda g_0 (Y,Y) = \half \mcl{L}_{X_0} g_0 (Y,Y).
\end{align*}
Since the left hand side of the above equation is left-invariant, it suffices to show that the vector field $X_0$ grows linearly at the identity element $e\in S$. Consider now a geodesic
\[ \gamma:=\gamma(s):[0,\epsilon) \longrightarrow S, \quad
\gamma(0)=e, \quad
g_0(\dot\gamma,\dot\gamma) = 1. \]
Then
\begin{align*}
\rc(g_0)(\dot\gamma, \dot\gamma) - \lambda 
& = \half \mcl{L}_{X_0}(\dot\gamma,\dot\gamma)\\
& = \half g_0(\nabla_{X_0} \dot\gamma,\dot\gamma) \\
%& = \nabla_{\dot\gamma} X_0 (\dot\gamma)\\
& = \nabla_{\dot\gamma} \left(g_0(X, \dot\gamma)\right) - g_0\left(X, \nabla_{\dot\gamma}\dot\gamma\right)\\
& = \frac{d}{ds}\left( g_0(X,\dot\gamma)\right).
\end{align*}
From this, there must exist at least one direction $\dot{\tilde\gamma}\in T_e S$ such that $\frac{d}{ds}\left( g_0(X,\dot{\tilde\gamma})\right)\neq 0$ for otherwise the metric $g_0$ would be Einstein.  Hence, $X_0$ grows linearly.
\end{proof}

Moreover, we have the following characterization of the vector field $X_0$.
\begin{lemma}\label{X_0}
There exists a coordinate chart on the solvmanifold $S$ such that $X_0 = d_i x^i \p_i$, where $d_1,\dots,d_n$ are the eigenvalues of the soliton derivation $D$.  
\end{lemma}

\begin{proof}
As $S$ is diffeomorphic to $\mbb{R}^n$, we think of $S$ as $(\mbb{R}^n,\cdot_S)$, i.e., $\mbb{R}^n$ with a different group structure.  Let $\mbb{R}^n$ have coordinates $(x^i)$, and let $\mfr{s}$ have an orthonormal basis $\{E_i\}$ in which $\rc$ and $D$ are diagonal (these are simultaneously diagonalizable by \eqref{eq:alg-soliton}).  We have two bases of $\mfr{s} = T_e \mbb{R}^n$, $\{\partial_i\}$ and $\{E_i\}$, that extend to global frame fields, but only the latter is left-invariant.  We wish to express the left-invariant frame field $\{E_i\}$ in terms of the coordinate vector fields, and this can be done with any diffeomorphism $F : \mbb{R}^n \rightarrow \mbb{R}^n$ such that $dF|_p : \partial_i|_p \mapsto E_i|_p$.  

As described in Section 5 of \cite{LauretWill2013}, the evolution of the algebraic soliton inner product $\langle \cdot,\cdot \rangle_0$ under Ricci flow (on the Lie algebra level) is $\langle \cdot,\cdot \rangle_t = \langle P(t) \cdot,\cdot \rangle_0$, where $P(t) : \mfr{s} \rightarrow \mfr{s}$ is a positive-definite automorphism.  Explicitly,
\[ P(t) = (-2\lambda t + 1) \exp[ \lambda^{-1} \log(-2\lambda t+1) D ]. \]
We want to describe the evolution of the Ricci soliton on the Lie group level, so we isolate the non-scaling part of $P(t)$ and take the square root of its inverse:
\[ \varphi(t) := \exp[ (-2\lambda)^{-1} \log(-2\lambda t+1) D ] = \mrm{diag}(\dots,(-2\lambda +1)^{-d_i/2\lambda},\dots). \]
Now, as $\varphi(t) : \mfr{s} \rightarrow \mfr{s}$ is an automorphism and $S$ is simply connected, we can integrate it to get a unique automorphism $\Phi(t) : S \rightarrow S$ that generates the soliton vector field.  This satisfies $d \Phi(t)|_e = \varphi(t)$, and the Ricci flow for the left-invariant metric is now
\[ g(t) = (-2\lambda t+1) \Phi(t)^* g_0. \]
With the map $F$ and the expression for $\varphi$, it is not hard to see that 
\[ \Phi(t)^i = (-2\lambda t+1)^{-d_i/2\lambda} x^i, \]
and differentiating in $t$ gives the soliton vector field:
\[ X_0 = \frac{d}{dt}\Big|_{t=0} \Phi(t) = \sum_{i} d_i x^i \partial_i. \qedhere \]
\end{proof}

We still refer to the right hand side of the curvature-normalized Ricci-DeTurck flow \eqref{eq:cnrdf} as $Q_{g_0}(g)g$.

\begin{lemma}\label{lem:coef-soliton}
If we express $Q_{g_0}(g)g$ in terms of the first and second derivatives of $g$ in local coordinates, then
\begin{align}\label{eq:coef-soliton}
(Q_{g_0}(g)g)_{ij}& =a\left(g_0(x),g_0^{-1}(x),g(x),g^{-1}(x)\right)_{ij}^{kl pq}\p_p\p_q g_{kl}\\
&\quad + b\left(g_0(x),g_0^{-1}(x),\p g_0(x),g(x),g^{-1}(x),\p g(x), X_0\right)_{ij}^{kl p}\p_p g_{kl} \notag \\
&\quad\quad +c\left(g_0^{-1}(x),\p g_0(x),\p^2 g_0(x),g(x)\right)_{ij}^{kl}g_{kl}, \notag
\end{align}
where $X_0$ is the vector field in the Ricci soliton equation. The coefficients $a, b, c$ are analytic functions of their arguments and depend smoothly on $x\in M$.
\end{lemma}

\begin{proof}
Given the computations of Lemma 5.1 in \cite{Wu2013}, we only need to determine the contributions to $a$, $b$, and $c$ from the term $\mcl{L}_{X_0} g$.  Dropping the 0 subscript on $X_0$, it is easy to see that 
\begin{align*} 
(\mcl{L}_X g)_{ij} 
&= \nabla_i X_j + \nabla_j X_i \\
%&= \nabla_i (g_{jk} X^k) + \nabla_j (g_{ik} X^k) \\
%&= g_{jk} \nabla_i X^k + g_{ik} \nabla_j X^k \\
&= g_{jk} (\partial_i X^k + X^l \Gamma_{il}^k) 
 + g_{ik} (\partial_j X^k + X^l \Gamma_{jl}^k) \\
&= g_{jk} \partial_i (d_k x^k) + g_{jk} (d_l x^l) \half g^{km} (\partial_i g_{l m} + \partial_l g_{im} - \partial_m g_{il}) \\
&\quad + g_{ik} \partial_j (d_k x^k) + g_{ik} (d_l x^l) \half g^{km} (\partial_j g_{l m} + \partial_l g_{jm} - \partial_m g_{jl}) \\ 
%&= + d_k g_{jk} \delta_i^k + \half (d_\ell x^\ell) \delta_j^m (\partial_i g_{\ell m} + \partial_\ell g_{im} - \partial_m g_{i\ell}) 
% + d_k g_{ik} \delta_j^k + \half (d_\ell x^\ell) \delta_i^m (\partial_j g_{\ell m} + \partial_\ell g_{jm} - \partial_m g_{j\ell}) \\
%&= + d_i g_{ij} + \half (d_\ell x^\ell) (\partial_i g_{\ell j} + \partial_\ell g_{ij} - \partial_j g_{i\ell}) 
% + d_j g_{ij} + \half (d_\ell x^\ell) (\partial_j g_{\ell i} + \partial_\ell g_{ij} - \partial_i g_{j\ell}) \\
&= (d_i + d_j) g_{ij} + d_k x^k \partial_k g_{ij},
\end{align*}
using Lemma \ref{X_0}.  This means that the only new contributions to $c$ are contants terms involving the $d_i$, and the only new contributions to $b$ are the terms $d_i x^i$, which are linear in $x$.
\end{proof}

This shows that in the soliton case, $Q_{g_0}(g)g$ is a second-order elliptic operator with unbounded first-order coefficient. We will consider perturbations supported on $B_R$ and vanishing on $\p B_R$ for some $R>0$ so that the operator has bounded coefficients, and we can apply the maximal regularity theory to deduce dynamical stability from strict linear stability.

Fix $0<\sigma<\rho<1$ and $R>0$, we now consider the following sequence of densely embedded spaces:
\begin{align*}
\begin{array}{ccccccc}
\mathbb X_1 & \overset{d}\hooklongrightarrow & \mathbb E_1 & \overset{d}\hooklongrightarrow &\mathbb X_0 & \overset{d}\hooklongrightarrow & \mathbb 
E_0\\
\|	& & \| & & \| & & \| \\
\mathfrak h^{2+\rho}(B_R) & & \mathfrak h^{2+\sigma}(B_R) & &\mathfrak h^{0+\rho}(B_R)& & \mathfrak h^{0+\sigma}(B_R)\\
\end{array}
\end{align*}
For fixed $\frac{1}{2}\leq\beta<\alpha<1-\frac{\rho}{2}$, we still set
\begin{align*}
\mathbb X_\alpha:=(\mathbb X_0,\mathbb X_1)_\alpha
\quad \text{and} \quad
\mathbb X_\beta:=(\mathbb X_0,\mathbb X_1)_\beta.
\end{align*}
For $0<\epsilon\ll 1$ to be chosen in Lemma \ref{analytic2}, we will let
\begin{align*}
 \mathbb{G}_\beta: = \{ g \in \mathbb{X}_\beta : g > \epsilon g_0 \}
\quad \text{and} \quad
 \mathbb{G}_\alpha : = \mathbb{G}_\beta \cap \mathbb{G}_\alpha.
\end{align*}
The linear operators $L_{g_0}$ and $\hat L_{g_0}$ are defined analogously as in Subsection \ref{subsec:dynamical-soliton}.

We note that Lemma \ref{analytic} holds true when $g_0$ is a Ricci soliton metric. Indeed, it is straightforward to verify that the Schauder estimates and hence the rest of the proof of Lemma \ref{analytic} go through in the Ricci soliton case. So we have the following lemma.

\begin{lemma}\label{analytic2}
 Let $g_0$ be an algebraic Ricci soliton. The operator $\hat L_{g_0}$ is sectorial, and there exists $\epsilon >0$ in the definition of $\GG_\beta$ such that for each $\hat g \in \GG_\alpha$, $\hat L_{\hat g}$ generates a strongly continuous analytic semigroup on $\mathcal L(\EE_0,\EE_0)$.
\end{lemma}

\begin{remark}
Lemma \ref{analytic2} is crucial for the application of Simonett's stability theorem (see condition (4) in Theorem \ref{stab}). When $g_0$ is a soliton, the operator $\mbf{L}$ (or $L_{g_0}$, or $\hat{L}_{g_0}$) is \emph{not} self-adjoint in any weighted $L^2$-space, and for this reason the authors in \cite[Section 5]{GuentherIsenbergKnopf2006} could prove that $\mbf{L}$ generates a $C_0$, but not analytic, semigroup, preventing them from establishing the dynamical stability of a soliton fixed point. However, the non-self-adjointness of $\mbf{L}$ is not an issue in this paper since we never use any weighted $L^2$-norm when proving Lemma \ref{analytic2}.
\end{remark}

\begin{proof}[Proof of Theorem \ref{thm:stab-soliton}]
Since the Ricci soliton $g_0$ is a strictly linearly stable fixed point, we apply Simonett's stability theorem (Theorem \ref{stab}) to conclude the dynamic stability. We omit the details since they resemble that for Einstein metrics on compact manifolds, see for example \cite{GuentherIsenbergKnopf2002}.
\end{proof}

\appendix

\section{Stability Theorem}\label{appxB}
We use the following version of Simonett's Stability Theorem. See \cite{GuentherIsenbergKnopf2002,Knopf2009} for a more general version of the theorem, and \cite{Simonett1995} for the most general statement.
\begin{thm}[Simonett]\label{stab}
Assume the following conditions hold:
\bn
\item[(1)] $\XX_1\overset{d}\hooklongrightarrow \XX_0$ and $\EE_1\overset{d}\hooklongrightarrow \EE_0$ are continuous and dense inclusions of Banach spaces. For fixed $0<\beta<\alpha<1$, $\XX_\alpha$ and $\XX_\beta$ are continuous interpolation spaces corresponding to the inclusion $\XX_1\overset{d}\hooklongrightarrow \XX_0$.
\item[(2)] Let
\begin{align}\label{eq:quasiparabeq}
\p_t g = Q(g)g
\end{align}
be an autonomous quasilinar parabolic equation for $t\geq 0$, with $Q(\cdot)\in C^k({\GG_\beta},\mathcal L(\XX_1,\XX_0))$ for some positive integer $k$ and some open set $\GG_\beta \subset \XX_\beta$.
\item[(3)] For each $\hat g\in\GG_\beta$, the domain $D(Q(\hat g))$ contains $\XX_1$, and there exists an extension $\hat Q(\hat g)$ of $Q(\hat g)$ to a domain $D(\hat Q(\hat g))$ containing $\EE_1$.
\item[(4)] For each $\hat g\in \GG_\alpha:=\GG_\beta\cap \XX_\alpha$, $\hat Q(\hat g)\in \mathcal L(\EE_1,\EE_0)$ generates a strongly continuous analytic semigroup on $\mathcal L(\EE_0,\EE_0)$.
\item[(5)] For each $\hat g\in \GG_\alpha$, $Q(\hat g)$ agrees with the restriction of $\hat Q(\hat g)$ to the dense subset $D(Q(\hat g))\subset\XX_0$.
\item[(6)] Let $(\EE_0,D(\hat Q(\cdot)))_\theta$ be the continuous interpolation space. Define the set 
\[ (\EE_0, D(\hat Q(\cdot)))_{1+\theta}:=\{x\in D(\hat Q(\cdot)):D(\hat Q(\cdot))(x)\in (\EE_0, D(\hat Q(\cdot))_\theta\}, \]
endowed with the graph norm of $\hat Q(\cdot)$ with respect to $(\EE_0,D(\hat Q(\cdot)))_\theta$. Then $\XX_0\cong (\EE_0,D(\hat Q(\cdot)))_\theta$ and $\XX_1\cong (\EE_0,D(\hat Q(\cdot)))_{1+\theta}$ for some $\theta\in(0,1)$.
\item[(7)] $\EE_1\overset{d}\hooklongrightarrow \XX_\beta \overset{d}\hooklongrightarrow \EE_0$ with the property that there are constants $C>0$ and $\theta\in(0,1)$ such that for all $x\in\EE_1$, one has
\begin{align*}
 \Vert x\Vert_{\XX_\beta} \leq C\Vert x \Vert_{\EE_0}^{1-\theta} \Vert x \Vert_{\EE_1}^\theta.
\end{align*}
\en

For each $\alpha\in(0,1)$, let $g_0\in\GG_\alpha$ be a fixed point of equation \eqref{eq:quasiparabeq}. Suppose that the spectrum of the linearized operator $DQ|_{g_0}$ is contained in the set $\{z\in\CC:\Re(z)\leq -\epsilon\}$ for some constant $\epsilon >0$. Then there exist constants $\omega\in(0,\epsilon)$ and $d_0, C_\alpha>0$, $C_\alpha$ independent of $g_0$, such that for each $d\in(0,d_0]$, one has
\begin{align*}
 \Vert \tilde g(t)-g_0 \Vert_{\XX_1} \leq \frac{C_\alpha}{t^{1-\alpha}} e^{-\omega t}\Vert \tilde g(0)-g_0\Vert_{\XX_\alpha}
\end{align*}
for all solutions $\tilde g(t)$ of equation \eqref{eq:quasiparabeq} with $\tilde g(0)\in B(\XX_\alpha,g_0,d)$, the open ball of radius $d$ centered at $g_0$ in the space $\XX_\alpha$, and for all $t\geq 0$.
\end{thm}

%\acknowl{Both authors thank Dan Knopf for helpful conversations and appreciate the hospitality of the Park City Mathematics Institute where part of this research was conducted.}

%%%%%%%%%%%%%%%%%%%%%%%%%

%%%%%%%%%%%%%%%%%%%%%%%%%

%\bibliography{hwu_bib,mbw_refs}


%%%%%%%%%%%%%%%%%%%%%%%%%
\end{document}